\documentclass{amsart}
\usepackage{mathrsfs}
\usepackage{mathtools} 
\usepackage{calligra}
\usepackage{amsthm}
\usepackage{amscd}
\usepackage{amssymb}
\usepackage{latexsym}
\usepackage{graphicx}
\usepackage{stmaryrd}
\usepackage{xypic}
\xyoption{curve}
\usepackage{nicefrac}
\usepackage{wasysym}
\usepackage{enumitem}
\usepackage{setspace}
\usepackage{ifthen}
\usepackage{verbatim}
\usepackage{eufrak}
\usepackage{relsize} 
\usepackage{tikz-cd}
\usepackage[margin=1.5in]{geometry}

\input{notation.tex} 

\newtheorem{thm}{Theorem}[section]

\newtheorem{lem}[thm]{Lemma}

\newtheorem*{lem*}{Lemma}

\theoremstyle{definition}
\newtheorem{defn}[thm]{Definition}

\newtheorem{eg}[thm]{Example} 

\newcommand{\sline}{{\Xi}}

\theoremstyle{plain}
\newcommand{\THMMonoidalDirealization}{Theorem \cite[5.23]{krishnan2015cubical}}
\newtheorem*{thm:monoidal.direalization}{\THMMonoidalDirealization}
\newcommand{\THMDiEmbed}{Theorem \cite[2.5]{fernandes2007classification}}
\newtheorem*{thm:diembed}{\THMDiEmbed}

\title{Triangulations of conal manifolds}
\author{Sanjeevi Krishnan}

\begin{document}
\begin{abstract}
  Casual structure can take the form of cone bundles on a manifold, more general local preorders on a topological space, or simplicial orientations implicit in a simplicial set.
	This note takes a triangulation of a conal manifold $M$ to mean an isomorphism between $M$ and the locally preordered geometric realization of a simplicial set.
	Such triangulations completely encode causal and topological structure combinatorially.  
	This note characterizes the triangulability of closed conal manifolds as the fibrewise freeness and generativity of the cone bundle.
\end{abstract}
\maketitle
\tableofcontents
\addtocontents{toc}{\protect\setcounter{tocdepth}{1}}

\section{Introduction}
Discrete approximations of phase spaces can give tractable methods for analyzing physical systems.  
Those discrete approximations often capture both the topology of the phase space and some of the causal structure relating different states.  
An example is a triangulation of spacetime into a partially oriented simplicial complex \cite{ambjorn2001dynamically,regge2000discrete}; incidence relations between simplices capture the topology, while the partial orientations capture some causal structure.
This note considers an alternative type of triangulation [Definition \ref{defn:triangulable}] for more general \textit{conal manifolds} and even more general locally preordered spaces, where partially oriented simplicial complexes are replaced with \textit{simplicial sets} and all topological $n$-simplices have exactly the same causal structure.
On one hand, the only $n$-spacetimes that admit such triangulations occur in dimensions $n=0,1,2$. 
On the other hand, the abstract simplicial sets in these triangulations completely capture both topological and casual structure.
This note characterizes triangulability for closed conal manifolds [Theorem \ref{thm:triangulable}].  
In particular, all closed timelike surfaces are triangulable [Figure \ref{fig:compact.timelike.surfaces}].
Section \S\ref{sec:ditopology} recalls definitions of \textit{conal manifolds} and more general \textit{streams}.  
Section \S\ref{sec:simplicial} recalls the definition of \textit{simplicial sets}.
Section \S\ref{sec:triangulations} recalls the definition of \textit{stream realizations $|S_*|$} of simplicial sets $S_*$ and defines a \textit{triangulation} of a conal manifold to be a stream isomorphism of the form $|S_*|\cong M$.  
For example, the reals $\R$ with the unique fibrewise positive and non-trivial cone bundle admits a triangulation $|\Xi_*|\cong\R$ into an infinite $1$-dimensional simplicial set, the directed graph $\Xi_*=\cdots\ra\bullet\ra\bullet\ra\cdots$.

\begin{thm}
  \label{thm:triangulable}
	The following are equivalent for a closed conal manifold $M$.
	\begin{enumerate}
		\item\label{item:triangulability} $M$ admits a triangulation
		\item\label{item:cone.bundle.nice} The fibers of $T^+M$ are generating and free.
		\item\label{item:manifold.compact} For each connected component $M_x$ of $M$, there exists a subgroup $G_x\leqslant\Sigma_n\rtimes\Z^n$ with $|\sline_*^n/G_x|\cong M_x$.
  \end{enumerate}
\end{thm}

Triangulability in this sense and closedness together thus constitute a rigid condition on conal manifold.  
Spacetimes and other phase spaces of interest need not be triangulable but admit arbitrary approximations by finer and finer meshes of triangulable locally preordered subspaces.  
The conal manifolds identified by the theorem can be regarded as toy models for understanding how to encode conal geometry into simplicial sets in this more generalized sense. 

\begin{figure}
	\includegraphics[width=4in,height=1.5in]{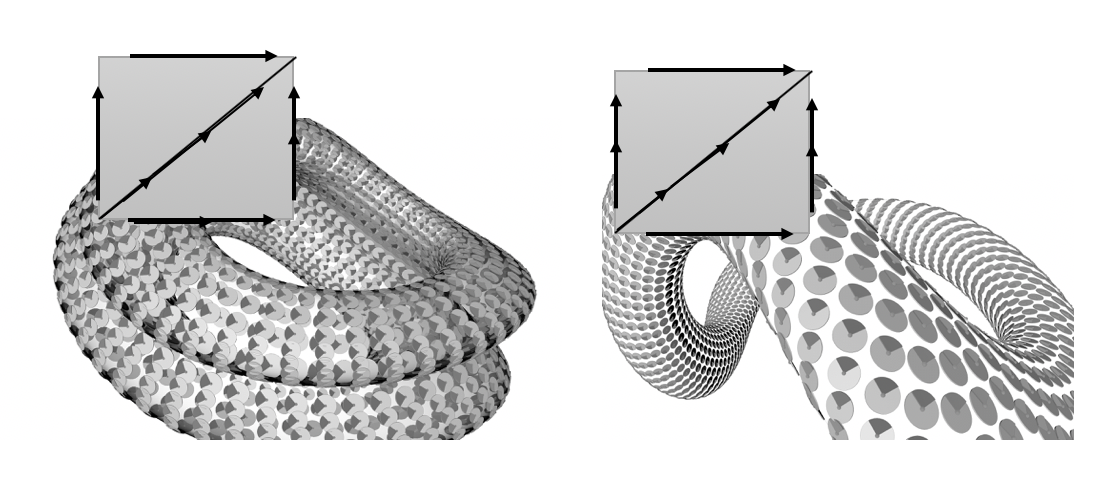}
  \caption{
    {\bf Conal manifolds}
    \textit{Conal manifolds}, smooth manifolds whose tangent spaces are all equipped with convex cones, naturally encode state spaces of processes under some causal constraints.
    The time-oriented Klein bottle (left) and time-oriented torus (right) depicted above are examples of conal manifolds that admit triangulations [Theorem \ref{thm:triangulable}], which encode casual and topological structure combinatorially.
    }
  \label{fig:compact.timelike.surfaces}
\end{figure}

\addtocontents{toc}{\protect\setcounter{tocdepth}{2}}

\section{Directed spaces}\label{sec:ditopology}
The causal structure on spacetime can be directly described without reference to the given Lorentzian metric.
A single \textit{causal preorder} on the underlying set of a past-distinguishing and future-distinguishing spacetime, or more generally a \textit{cone bundle} of causal tangent vectors on the underlying smooth manifold of a general spacetime, determines the entire causal conformal structure \cite[Theorem 1]{malament1977class}.
A \textit{circulation} generalizes the data of a cone bundle beyond the smooth setting.  
This section shows how to regard both \textit{conal manifolds}, smooth manifolds equipped with cone bundles, and topological simplices equipped with continuous lattice operations, as examples of \textit{streams}, spaces equipped with circulations.

\subsection{Preordered spaces}
A \textit{preorder} on a set $X$ is a transitive and reflexive relation on $X$.  
A \textit{preordered space} is a space $X$ equipped with a preorder, denoted by $\leqslant_X$, on its underlying set.
A \textit{monotone map} is a continuous function $f:X\ra Y$ between preordered spaces $X$ and $Y$ such that $\phi(x)\leqslant_Yf(y)$ whenever $x\leqslant_Xy$.  
The reader is referred elsewhere \cite{gierz2003continuous,nachbin1976topology} for a primer on order and topology.
Regard the topological $n$-simplex
$$\nabla[n]=\{{\mathbf t}\in\R^{n}\;|\;1\geqslant t_1\geqslant t_2\geqslant\cdots\geqslant t_n\geqslant 0\}\subset\R^n$$
as a preordered space with $x\leqslant_{\nabla[n]}y$ if $x_i\leqslant y_i$ for all $i=1,\ldots,n$. 

\begin{eg}
  The preordered space $\nabla[1]$ is the unit interval with the usual ordering.
\end{eg}

Examples of monotone maps are the functions
$$\delta_{i}:\nabla[n-1]\ra\nabla[n]\quad\sigma_i:\nabla[n+1]\ra\nabla[n]$$
defined by the following rules:
\begin{align*}
	\delta_0(x_1,x_2,\ldots,x_n)&=(0,x_1,x_2,\ldots,x_n)\\
  \delta_i(x_1,x_2,\ldots,x_n)&=(x_1,x_2,\ldots,x_i,x_{i},x_{i+1},\ldots,x_n) & 0<i\leqslant n\\
	\sigma_0(x_1,x_2,\ldots,x_{n+1})&=(x_2,\ldots,x_n)\\
	\sigma_n(x_1,x_2,\ldots,x_{n+1})&=(x_1,x_2,\ldots,x_n)\\
	\sigma_i(x_1,x_2,\ldots,x_{n+1})&=(x_1,\ldots,x_{i},x_{i+2},\ldots,x_n) & 0<i<n
\end{align*}

Consider a real topological vector space $V$.
A \textit{convex pointed cone} in $V$ is a subset of $V$ closed under vector addition and scalar multiplication by the set $\R_{\geqslant 0}$ of non-negative reals.
A convex pointed cone $C$ in a vector space $V$ is \ldots
\begin{enumerate}
	\item \textit{generating} if $C$ spans $V$ as a vector space
  \item \textit{free} if it is generated by a linearly independent set of vectors by vector addition and scalar multiplication by non-negative reals
\end{enumerate}

\begin{eg}
	The set $\R_{\geqslant 0}$ is a free and generating convex cone in $\R$.  
\end{eg}

An \textit{ordered topological vector space} is a real topological vector space $V$ equipped with a convex pointed cone $V^+$.
Regard $\R^n$ as an ordered topological vector space with 
$$(\R^n)^+=\R_{\geqslant 0}^n.$$ 
An ordered topological vector space will henceforth be regarded as a preordered space with $v_1\leqslant_Vv_2$ if $v_2-v_1\in V^+$.  

\subsection{Conal manifolds}\label{subsec:conal.manifolds}
A \textit{cone bundle} on a smooth manifold $\manifold$ is a subbundle
$$T^+\manifold\subset T\manifold$$
of the tangent bundle $TM$ that is fibrewise a pointed convex cone in $TM$.
A \textit{conal manifold} is a smooth manifold equipped with a cone bundle.
The reader is referred elsewhere for a primer on conal geometry \cite{khavkine2012characteristics}.
Each finite-dimensional ordered topological vector space $V$ will be regarded as a conal manifold whose cone bundles have fibers isomorphic to $V^+$.
In this manner, $\R^n$ and tangent spaces of conal manifolds will henceforth be regarded as conal manifolds.
More generally, a time-oriented Lorentzian manifold can be regarded as a conal manifold with cone bundle consisting of all non-spacelike tangent vectors.

\begin{eg}
  \label{eg:spacetimes}
  The cone bundle of a time-oriented Lorentzian manifold is \ldots
  \begin{enumerate}
    \item \ldots fibrewise generating
    \item \ldots fibrewise free if and only if the dimension of the manifold is less than $3$.
  \end{enumerate}
\end{eg}

Each conal manifold $M$ is equipped with a \textit{conal preorder} $\leqslant_M$ on its underlying set of points such that $x\leqslant_My$ if and only if there is a smooth path from $x$ to $y$ whose derivatives all lie in the cone bundle.
A \textit{conal map} is a smooth map between conal manifolds whose differential maps cone bundle into cone bundle.
A \textit{conal isomorphism} is a conal map admitting an inverse conal map.
A \textit{conal isometry} between conal manifolds equipped with Riemannian metrics is a conal map that is also an isometry.
Implicitly regard $\R^n$ as equipped with the standard Euclidean metric.  

\begin{lem}
  \label{lem:conal.embeddings}
	Every conal isometry $e:\R^n\cong\R^n$ factors as a composite
	$$e=\tau\sigma:\R^n\cong\R^n$$
	of a coordinate permutation $\sigma:\R^n\cong\R^n$ followed by an additive translation $\tau$.
\end{lem}

The following proof uses the fact that every isometry between Euclidean spaces is an affine transformation.  
An alternative proof uses the rigidity of embeddings of the topological lattice $\I^n$ into $\R^n$ \cite[Theorem 2.5]{fernandes2007classification}.  

\begin{proof}
	Every isometry is an affine transformation and hence the composite of a linear transformation followed by an additive translation, a conal isomorphism.
	It therefore suffices to consider the case $e$ a linear isomorphism.
	In that case, $e$ restricts and corestricts to an automorphism on $(\R^n)^+$ of convex cones by $e$ an isomorphism of ordered topological vector spaces.
	Thus $e$ preserves unit-length extremal vectors in $(\R^n)^+$.
	Thus $e$ permutes standard basis vectors.
\end{proof}

Smooth compact conal manifolds with fibrewise free and generating cone bundle are especially rigid.
Regard $\R^n$ as admitting an action of $\Z^n$ by additive translations, $\Sigma_n$ by coordinate permutations, and hence a semidirect product $\Z^n\semidirectproduct\Sigma_n$ by composites of such conal isomorphisms.

\begin{lem}
  \label{lem:rigidity}
  The following are equivalent for a connected compact conal $n$-manifold $M$.
  \begin{enumerate}
    \item $T^+M$ is fibrewise free and fibrewise generating
    \item There exist subgroup $G$ of $\Z^n\rtimes\Sigma_n$ and conal isomorphism $M\cong\R^n/G$.
  \end{enumerate}
\end{lem}
\begin{proof}
  Endow $M$ with a Riemannian metric compatible with its smooth structure.
  Each fiber of $TM$ is isomorphic as a conal manifold to $\R^n$.
  Hence each linear conal isometry $T_mM\cong T_mM$ is a coordinate permutation [Lemma \ref{lem:conal.embeddings}].
  Thus $M$ is flat because its Riemannian holonomy, a subgroup of the discrete group $\Sigma_n$, is discrete.

  Let $u:U\ra M$ be the universal covering of $M$.
  Endow $U$ with Riemannian metric and cone bundle pulled back from $M$ along $u$.
  Fix $x\in U$.
  Let $\exp_x$ denote the exponential map $T_xU\ra U$.
  The underlying Riemannian manifold of $U$ is $\R^n$ by $M$ connected, compact, and flat; thus $\exp_x$ is a diffeomorphism.
	Then $\exp_x$ embeds an open neighborhood of each point, regarded as a submanifold of $T_xU$ whose cone bundle is a restriction of the cone bundle on $T_xU$, into $U$.
  Thus $\exp_x$ is a conal isomorphism.
  Thus $u$ can taken to be of the form
  \begin{equation}
    \label{eqn:conal.cover}
    u:\R^n\ra M.
  \end{equation}

  Let $G$ be the group of deck transformations of (\ref{eqn:conal.cover}).
	Each $g\in G$, an isometric conal isomorphism $\R^n\cong\R^n$, factors into a coordinate permutation followed by an additive translation [Lemma \ref{lem:conal.embeddings}].  
  All of these translation components of elements in $G$ can be taken to be translations by integer vectors because $G$ is finitely generated by $M$ compact.  
	Thus $G$ can be taken to be a subgroup of $\Z^n\rtimes\Sigma_n$.
\end{proof}

In fact, every quotient of $\R^n$ by a cocompact and fixed-point free subgroup $G$ of $\Z^n\rtimes\Sigma_n$ gives rise to just such a compact conal manifold.
The only possibilities when $n=2$ are the time-oriented torus and time-oriented Klein bottle [Figure \ref{fig:compact.timelike.surfaces}].
For the torus, $G=\Z^2$.
For the Klein bottle, $G\cong\Z*_{2\Z}\Z$ is presented by two conal automorphisms $\alpha:(x,y)\mapsto(y+1,x)$ and $\beta:(x,y)\mapsto(y,x+1)$ on $\R^2$ subject to a single relation $\alpha^2\beta^{-2}$.

\subsection{Streams}\label{subsec:streams}
A \textit{circulation} on a topological space $X$ is a function
$$\leqslant:U\mapsto\;\leqslant_U$$
assigning to each open subset $U\subset X$ a preorder $\leqslant_U$ on $U$ such that $\leqslant$ sends the union of a collection $\mathcal{O}$ of open subsets of $X$ to the preorder with smallest graph containing the graph of $\leqslant_U$ for each $U\in\mathcal{O}$ \cite{krishnan2009convenient}.
A \textit{stream} is a space equipped with a circulation on it.
A continuous function $f:X\ra Y$ of streams is a \textit{stream map} if $f(x)\leqslant_Uf(y)$ whenever $x\leqslant_{f^{-1}U}y$ \cite{krishnan2009convenient}.
A \textit{stream isomorphism} is a stream map $f:X\ra Y$ for which there exists a stream map $g:Y\ra X$ with $gf=\id_X$ and $fg=\id_Y$.
The reader is referred elsewhere \cite{krishnan2009convenient} for the point-set theory of streams.  

\begin{eg}
  \label{eg:initial.circulations}
  Every topological space admits an \textit{initial circulation} $\leqslant$ defined by
  \begin{equation*}
    x\leqslant_Uy\iff x=y\in U
  \end{equation*}
\end{eg}

\begin{eg}
  \label{eg:conal.streams}
  Each conal manifold $\manifold$ admits a circulation assigning each open submanifold $U\subset M$, regarded as a conal manifold with the restricted cone bundle, the conal preorder $\leqslant_U$ on $U$.
  Thus all conal manifolds will henceforth be regarded as streams.
  On certain conal manifolds like $\R^n$, there exists a unique circulation sending the entire conal manifold to its conal preorder \cite[Lemma 4.2]{krishnan2009convenient}.
	Uniqueness implies that a continuous functions$f:X\ra\R^n$ from a stream $X$ is a stream map if $f(x)\leqslant_{\R^n}f(y)$ whenever $x\leqslant_Xy$.
\end{eg}

Explicit constructions of circulations are often cumbersome.
Henceford regard $\R^n$ and $\nabla[n]$ as streams as described in the following examples.

\begin{eg}
  \label{eg:conal.streams}
	There exists a unique circulation turning the underlying space of $\nabla[n]$ into a stream whose circulation sends the entire space to the given preorder $\leqslant_{\nabla[n]}$ on $\nabla[n]$ \cite[Lemmas 4.2, 4.4 and Example 4.5]{krishnan2009convenient}.
	Uniqueness implies that all monotone maps of the form
	$$\delta_i:\nabla[n]\ra\nabla[n-1]\quad\sigma_i:\nabla[n]\ra\nabla[n+1]$$
	are stream maps.
\end{eg}

Streams, like topological spaces, admit universal constructions such as products, disjoint unions, and quotients.
A disjoint union $\amalg_iX_i$ of a family of streams $X_i$ is an ordinary disjoint union of underlying spaces equipped with the unique circulation whose restriction to the topology of $X_j$ coincides with the circulation on $X_j$.
A quotient of a stream $X$ by an equivalence relation $\equiv$ on the underlying set of $X$ is the quotient space $X/\!\!\equiv$ with circulation assigning to each open subset $U/\!\!\equiv$ the transitive closure of the relation induced by the preorder $\leqslant_U$ on $U\subset X$.
A finite product $X\times Y$ of streams $X$ and $Y$ is the ordinary product of underlying spaces equipped with the unique circulation sending an open subset of the form $U\times V$ to the product of the preorders $\leqslant_U$ and $\leqslant_V$.  

\section{Simplicial sets}\label{sec:simplicial}
A \textit{simplicial set} is an abstract generalization of a simplicial complex with the extra structure of consistent choices of orientations on the simplices.  
Formally, a \textit{simplicial set} $S_*$ is a graded sequence $S_0,S_1,\ldots$ of sets equipped with functions
\begin{equation}
	\label{eqn:structure.maps}
	d_0,d_1,\ldots,d_n,d_{n+1}:S_{n+1}\ra S_{n}\quad s_0,s_1,\ldots,s_n:S_{n}\ra S_{n+1}
\end{equation}
for each $n=0,1,\ldots$, all satisfying all possible identities of the following form:
\begin{align*}
	d_id_j&=d_{j-1}d_i & i<j\\
	s_is_{j-1}&=s_js_{i} & i<j\\
	d_is_j&=s_{j-1}d_i & i<j\\
	d_is_j&=s_jd_{i-1} & i>j+1\\
	d_is_i&=d_{i+1}s_i=\id_{S_n}
\end{align*}

The reader is referred elsewhere \cite{friedman2012survey} for the general theory of simplicial sets, including their relationship with simplicial complexes.  
An $n$-simplex in a simplicial set $S_*$ is an element in $S_n$.  
An $n$-simplex $\theta$ is \textit{non-degenerate} if $n=0$ or $\theta$ does not lie in the image of a function of the form $s_i:S_{n-1}\ra S_n$.  
Intuitively for an $n$-simplex $\theta$, $d_i\theta$ is its $i$th face and $s_i\theta$ is an interpretation of $\theta$ as a degenerate $(n+1)$-simplex.
The \textit{dimension} of a simplicial set $S_*$ is the supremum over all $n=0,1,2,\ldots$ for which $S_*$ has a non-degenerate $n$-simplex.

\begin{eg}
	A $1$-dimensional simplicial set $S_*$ is determined by restricted functions
	$$d_1,d_0:S_1-s_0(S_0)\ra S_0,$$
	which can be respectively regarded as the source and target of a directed graph with vertex set $S_0$ and edge set $S_1-s_0(S_0)$ consisting of all non-degenerate $1$-simplices in $S_*$.  
\end{eg}

Let $\sline_*$ be the $1$-dimensional simplicial set with $\sline_0$ the set of integers, $\sline_1$ the set of all pairs of integers of the form $(i,i)$ and $(i,i+1)$, $d_1,d_0$ respectively defined by projection onto first and second coordinates, and $s_0$ defined by $s_0(i)=(i,i)$.  
As a directed graph,
$$\sline_*=\left(\cdots\ra\bullet\ra\bullet\ra\cdots\right)$$

An $n$-simplex $\theta$ in $S_*$ \textit{has vertex $v$} if $v=d_{i_1}d_{i_2}\cdots d_{i_n}\theta$ for some choie of natural numbers $i_1,i_2,\ldots,i_n$.
A simplicial set $S_*$ is \textit{locally finite} if, for each vertex $v$, there exists only finitetly many non-degenerate simplices having vertex $v$.
\textit{Products} and \textit{disjoint unions} of simplicial sets are straightforwardly defined degreewise.  
An \textit{action} of a group $G$ on a simplicial set $S_*$ is an action of $G$ on each of the sets $S_0,S_1,\ldots$ commuting with the structure maps (\ref{eqn:structure.maps}).  
For each simplicial set $S_*$ equipped with the action of a group $G$, the structure maps (\ref{eqn:structure.maps}) induce structure maps on the degreewise orbits $S_*/G$ turning $S_*/G$ into a simplicial set.  

\section{Triangulations}\label{sec:triangulations}
The \textit{stream realization} of a simplicial set $S_*$ is the stream
$$|S_*|=\quotient{\amalg_{n=0}^\infty\;S_n\times\nabla[n]}{\equiv},$$
where $S_n$ is regarded as a discrete topological space equipped with the initial circulation and $\equiv$ is the smallest equivalence generated by all relations of the form $(d_i\theta,x)\equiv(\theta,\delta_i(x))$ and $(s_i\theta,x)\equiv(\theta,\sigma_i(x))$.  
Intuitively, $|S_*|$ is a stream obtained by gluing together different copies of $\nabla[n]$, one for each $n$-simplex, according to the structure maps of $S_*$.  
Formally, stream realization preserves disjoint unions and quotients.  
Less formally, stream realization preserves certain products.

\begin{thm:monoidal.direalization}
	For all locally finite simplicial sets $A_*$ and $B_*$,
	$$|A_*\times B_*|\cong|A_*|\times|B_*|.$$
\end{thm:monoidal.direalization}

\begin{defn}
	\label{defn:triangulable}
  A \textit{triangulation} of a stream $X$ is a stream isomorphism of the form
  $$|S_*|\cong X$$
\end{defn}

An example is the following triangulation of $\R$.  

\begin{lem}
	\label{lem:1-triangulations}
  There exists a triangulation $|\Xi_*|\cong\R$ bijectively sending vertices to integers.  
\end{lem}
\begin{proof}
	For each $\theta\in\Xi_1$, define a monotone map and hence stream map
	$$\varphi_\theta:\nabla[1]\ra\R$$
	by the rule $\varphi_\theta(t)=(1-t)d_1\theta+td_0\theta$.
	The induced stream map 
	$$\amalg_{\theta\in\Xi_1}\varphi_{\theta}:\amalg_{\theta\in\Xi_1}\nabla[n]\ra\R$$
	passes to a stream map $\varphi:|\Xi_*|\cong\R$, a homeomorphism because $\R$ is the union of its closed intervals with the union topology.  
	Consider $t_1\leqslant_{\R}t_2$.  
	In the case $t_1,t_2$ lie in a closed interval of the form $[i,i+1]$ for an integer $i$, then $\varphi^{-1}(t_1)\leqslant_{|\Xi_*|}\varphi^{-1}(t_2)$ because $\varphi^{-1}_{(i,i+1)}(t_1)\leqslant_{\nabla[1]}\varphi^{-1}_{(i,i+1)}(t_2)$.
	In the general case, $\varphi^{-1}(t_1)\leqslant_{|\Xi_*|}\varphi^{-1}(t_2)$ by $\leqslant_{|\Xi_*|}$ a preorder and $\leqslant_{\R}$ the transitive closure of the restrictions of the preorder on closed intervals of the form $[i,i+1]$.  
	Thus $\varphi^{-1}$ is a monotone map of preordered spaces.
	Thus $\varphi$ is an isomorphism of preordered spaces and hence a stream isomorphism by uniqueness of the circulation on $\R$ sending all of $\R$ to the standard order.  
\end{proof}

A triangulation $|S_*|\cong M$ of a conal manifold $M$ forces the simplicial set $S_*$ to essentially be determined by the much simpler data of a \textit{presimplicial set/semisimplicial set/$\Delta$-complex} \cite{friedman2012survey} by local triviality of the cone bundle.   
However, presimplicial sets are inconvenient as an intrinsic formalism for combinatorial phase spaces because the analogue of  \cite[Theorem 5.23]{krishnan2015cubical} fails to hold for them.
This latter result is convenient in the proof of the main result.

\begin{proof}[proof of Theorem \ref{thm:triangulable}]
	The simplicial set $\sline_*^{n}$ admits a $\Sigma_n$ action defined by permutation of factors, $\Z^n$ action defined by addition of integer vectors, and hence a $(\Z^n\rtimes\Sigma_n)$-action defined by composites of the previous sorts of actions.  

	Assume (\ref{item:triangulability}).  
  Consider a triangulation $|S_*|\cong M$.  
	The dimension $n$ of $S_*$ is the dimension of $M$ because the underlying space of $S_*$ is the ordinary geometric realization of $S_*$. 
	Thus there exists a non-degenerate $n$-simplex $\theta$ in $S_n$.
  Let $U=\nabla[n]-\cup_{i=0}^n\delta_i(\nabla[n-1])$.
	The quotient stream map $\amalg_nS_n\times\nabla[n]\ra|S_*|$ restricts and costricts to an isomorphism from $\{\theta\}\times U$ to an open substream of $|S_*|$.  
	Then there exists a point in $|S_*|$ and hence also a point in $M$ admitting an open neighborhood, regarded as an open substream, isomorphic to $\R^n$.  
	Thus $T^+M$ has fibers isomorphic to $\R^n$ and hence (\ref{item:cone.bundle.nice}) by local triviality of the cone bundle on $M$.

	Assume (\ref{item:cone.bundle.nice}).
	It suffices to take $M$ connected because $|-|$ preserves disjoint unions.  
  In that case, take $M=\R^n/G$ for some discrete subgroup $G$ of $\Z^n\rtimes\Sigma_n$ without loss of generality [Lemma \ref{lem:rigidity}].

  Define $e_n:|S_*^n|\cong\R^n$ as follows.
	In the case $n=1$, $e_1$ can be taken to be a $\Z$-equivariant stream isomorphism [Lemma \ref{lem:1-triangulations}].
	Let $e_n$ be the induced stream isomorphism between $n$-fold products, of the desired form because $|-|$ preserves finite products of locally finite simplicial sets \cite[Theorem 5.23]{krishnan2015cubical}.
  Then $e_n$, $\Z^n$-equivariant by $e_1$ $\Z$-equivariant and $\Sigma_n$-equivariant by construction, is $G$-equivariant.
	Thus $M=\R^n/G\cong|\sline^n_*|/G$, isomorphic to $|\sline_*^n/G|$ because $|-|$ preserves quotients.
	Thus (\ref{item:manifold.compact}).

	Assume (\ref{item:manifold.compact}).
	We can decompose $M=\amalg_iM_i$ with each $M_i$ connected.
	For each $i$, there exists subgroup $G_i$ of $\Sigma_n\rtimes\Z^n$ with $M_i=\R^n/G_i$ [Lemma \ref{lem:rigidity}]
	Then $|\amalg_{i}(\sline_*^n/G_i)|=\amalg_i|\sline_*^n/G_i|\cong\amalg_iM_i=M$.
	Thus (\ref{item:triangulability}).
\end{proof}

\section{Acknowledgements}
The author is grateful to Robert Ghrist for conceiving of and producing the visualizations of conal manifolds behind the triangulations in Figure \ref{fig:compact.timelike.surfaces} and Ulrich Gelrach for helpful conversations on GR.  
This work was supported by AFOSR grant FA9550-16-1-0212.

\appendix
\addcontentsline{toc}{section}{Appendix}
\addtocontents{toc}{\protect\setcounter{tocdepth}{0}}


\bibliography{gv}{}
\bibliographystyle{plain}
\end{document}